\documentclass{amsart}

\newtheorem{theorem}{Theorem}[section]
\newtheorem{lemma}[theorem]{Lemma}
\newtheorem{proposition}[theorem]{Proposition}
\newtheorem{corollary}[theorem]{Corollary}

\theoremstyle{definition}
\newtheorem{definition}[theorem]{Definition}
\newtheorem{example}[theorem]{Example}

\theoremstyle{remark}
\newtheorem{remark}[theorem]{Remark}

\numberwithin{equation}{section}

\begin{document}

\title{$L^p$ regularity of least gradient functions}

\def\Xint#1{\mathchoice
   {\XXint\displaystyle\textstyle{#1}}%
   {\XXint\textstyle\scriptstyle{#1}}%
   {\XXint\scriptstyle\scriptscriptstyle{#1}}%
   {\XXint\scriptscriptstyle\scriptscriptstyle{#1}}%
   \!\int}
\def\XXint#1#2#3{{\setbox0=\hbox{$#1{#2#3}{\int}$}
     \vcenter{\hbox{$#2#3$}}\kern-.5\wd0}}
\def\ddashint{\Xint=}
\def\dashint{\Xint-}

\newcommand{\twopartdef}[4]
{
\left\{
		\begin{array}{ll}
			#1 & #2 \\
			#3 & #4
		\end{array}
	\right.
}

\newcommand{\threepartdef}[6]
{
	\left\{
		\begin{array}{lll}
			#1 & #2 \\
			#3 & #4 \\
			#5 & #6
		\end{array}
	\right.
}


\author{Wojciech G\'{o}rny}

\address{W. G\'{o}rny: Faculty of Mathematics, Informatics and Mechanics, University of Warsaw, Warsaw, Poland.}
\curraddr{}
\email{w.gorny@mimuw.edu.pl}
\thanks{}

\subjclass[2010]{35J20, 35J25, 35J75, 35J92}
\keywords{Least Gradient Problem, Anisotropy, $L^p$ regularity}

\date{\today}

\dedicatory{}

\commby{}

\begin{abstract}
It is shown that solutions to the anisotropic least gradient problem for boundary data $f \in L^p(\partial\Omega)$ lie in $L^{\frac{Np}{N-1}}(\Omega)$; the exponent is shown to be optimal. Moreover, the solutions are shown to be locally bounded with explicit bounds on the rate of blow-up of the solution near the boundary in two settings: in the anisotropic case on the plane and in the isotropic case in any dimension.
\end{abstract}

\maketitle

\section{Introduction}

Our main focus is the least gradient problem, which is the following minimisation problem

\begin{equation}\label{problem}\tag{LGP}
\min \{ \int_\Omega |Du|, \quad u \in BV(\Omega), \quad u|_{\partial\Omega} = f  \}.
\end{equation}
This problem was introduced in \cite{SWZ}, where the authors estabilish that for continuous boundary data, under a set of conditions on an open bounded set $\Omega \subset \mathbb{R}^N$ slightly weaker than strict convexity, a unique solution to Problem (\ref{problem}) exists and it is continuous up to the boundary. Recently, the authors of \cite{JMN} considered an anisotropic version of the least gradient problem:
\begin{equation}\label{aproblem}\tag{ALGP}
\min \{ \int_\Omega \phi(x, Du), \quad u \in BV(\Omega), \quad u|_{\partial\Omega} = f  \}.
\end{equation}
This type of problems appear as a dimensional reduction in the free material design, see \cite{GRS}, and conductivity imaging, see \cite{JMN}. 
In this paper, we follow the approach to this problem from the point of view of geometric measure theory, following \cite{BGG}, \cite{JMN}, and \cite{SWZ}. In particular, we understand the boundary condition in the sense of traces of $BV$ functions.

In both the isotropic and anisotropic least gradient problem, existence of solutions depends on the shape of $\Omega$. In particular, for continuous boundary data sufficient conditions are: for Problem \eqref{problem} strict convexity of $\Omega$, see \cite{SWZ}; for Problem \eqref{aproblem}, the barrier condition introduced in \cite{JMN}. As continuous boundary data are bounded, by a maximum principle we obtain an immediate $L^\infty$ bound on the solution. However, under suitable regularity assumptions on $\phi$, a recent article \cite{Gor4} gives existence of solutions to Problem \eqref{aproblem} also for unbounded boundary data, provided that their discontinuity set has Hausdorff measure zero. In this case, the direct method gives no regularity estimates for the solutions.

The paper is organised as follows: Section 2 provides the necessary background concerning anisotropic least gradient functions. Section 3 is devoted to proving the main result of this paper, i.e. Theorem \ref{thm:lpregularity}, which concerns $L^{\frac{Np}{N-1}}$ regularity of solutions to the least gradient problem for boundary data which lie in $L^p(\partial\Omega)$, using an argument based on the isoperimetric inequality. Moreover, in Example \ref{ex:optimal} we see that the exponent $\frac{Np}{N-1}$ is optimal. 

Let us stress that we do not discuss existence or uniqueness of minimisers; here, given a minimiser of Problem \eqref{aproblem}, we prove an estimate of its $L^{\frac{Np}{N-1}}$ norm. For this reason, we only assume $\Omega$ to be an open bounded set with Lipschitz boundary; we do not impose geometric assumptions on $\Omega$ sufficient to obtain existence of minimisers. However, we have an indirect assumption that the set $\Omega$ and the function $f$ support at least one solution to the anisotropic least gradient problem. 

In Section 4 we prove that solutions to the least gradient problem are locally bounded. This is done in two settings: firstly, in $\mathbb{R}^2$ in the anisotropic case, using a characterisation of one-dimensional integral currents; secondly, in the isotropic least gradient problem in any dimension, using the monotonicity formula for area-minimising boundaries.

\section{Preliminaries}

\subsection{Least gradient functions}

In this section, we recall the definition of least gradient functions on bounded domains and their basic properties.

\begin{definition}\label{dfn:lg}
Let $\Omega \subset \mathbb{R}^N$ be an open and bounded set. We say that $u \in BV(\Omega)$ is a function of least gradient, if for every $v \in BV(\Omega)$ compactly supported in $\Omega$ we have
\begin{equation*}
\int_\Omega |Du| \leq \int_\Omega |D(u + v)|.
\end{equation*}
In case when $\Omega$ has Lipschitz boundary, this is equivalent to the condition that $v \in BV_0(\Omega)$, see \cite[Theorem 2.2]{SZ0}. This equivalence is proved using an approximation with functions of the form $v_n = v \chi_{\Omega_n}$ for suitably chosen $\Omega_n$.
\end{definition}

\begin{definition}
We say that $u \in BV(\Omega)$ is a solution to Problem (\ref{problem}), if $u$ is a function of least gradient and the trace of $u$ equals $f$, i.e. for $\mathcal{H}^{N-1}-$almost every $x \in \partial\Omega$ we have 
$$ {\lim_{r \rightarrow 0^+} \,} \dashint_{B(x,r) \cap \Omega} |f(x) - u(y)| dy = 0.$$
\end{definition}

Now, we recall a classical theorem by Bombieri-de Giorgi-Giusti, which gives us a link between the function $u$ of least gradient and the regularity of its superlevel sets. Here and in the whole manuscript let us denote $E_t = \{ u \geq t \}$.

\begin{theorem}\label{tw:bgg}
(\cite[Theorem 1]{BGG}) Suppose that $\Omega \subset \mathbb{R}^N$ is open and let $u \in BV(\Omega)$ be a function of least gradient in $\Omega$. Then for every $t \in \mathbb{R}$ the set $E_t$ is minimal in $\Omega$, i.e. the function $\chi_{E_t}$ is of least gradient. \qed
\end{theorem}

Finally, as least gradient functions are $BV$ functions, they are defined up to a set of measure zero, we have to choose a proper representative if we want to state any pointwise regularity results. In this paper we deal with $L^p$ regularity, so at first glance it is not an issue; however, in the proofs in Section 4 we will use regularity of boundaries of area-minimising sets, so following \cite{SWZ} we employ the convention that a set of a bounded perimeter consists of all its points of positive density.

\subsection{Anisotropic formulation}

Firstly, we recall the notion of a metric integrand and $BV$ spaces with respect to a metric integrand. This entire subsection is based on the construction in \cite{AB}.

\begin{definition}
Let $\Omega \subset \mathbb{R}^N$ be an open bounded set with Lipschitz boundary. A continuous function $\phi: \overline{\Omega} \times \mathbb{R}^N \rightarrow [0, \infty)$ is called a metric integrand, if it satisfies the following conditions: \\
\\
$(1)$ $\phi$ is convex with respect to the second variable for a.e. $x \in \overline{\Omega}$; \\
$(2)$ $\phi$ is {1-}homogeneous with respect to the second variable, i.e.
\begin{equation*}
\forall \, x \in \overline{\Omega}, \quad \forall \, \xi \in \mathbb{R}^N, \quad \forall \, t \in \mathbb{R} \quad \phi(x, t \xi) = |t| \phi(x, \xi);
\end{equation*}
$(3)$ $\phi$ is bounded and {uniformly} elliptic in $\overline{\Omega}$, i.e.
\begin{equation*}
\exists \,  \lambda, \Lambda  > 0 \quad \forall \, x \in \overline{\Omega}, \quad \forall \, \xi \in \mathbb{R}^N \quad \lambda |\xi| \leq \phi(x, \xi) \leq \Lambda |\xi|.
\end{equation*}
These conditions apply to most cases considered in the literature, such as the classical least gradient problem, i.e. $\phi(x, \xi) = |\xi|$ (see \cite{SWZ}), the weighted least gradient problem, i.e. $\phi(x, \xi) = g(x) |\xi|$ (see \cite{JMN}), where $g \geq c > 0$, and $l_p$ norms for $p \in [1, \infty]$, i.e. $\phi(x, \xi) = \| \xi \|_p$ (see \cite{Gor1}).
\end{definition}

\begin{definition}
The polar function of $\phi$ is $\phi^0: \overline{\Omega} \times \mathbb{R}^N \rightarrow [0, \infty)$ defined as
\begin{equation*}
\phi^0 (x, \xi^*) = \sup \, \{ \langle \xi^*, \xi \rangle : \, \xi \in \mathbb{R}^N, \, \phi(x, \xi) \leq 1 \}.
\end{equation*} 
\end{definition}

\begin{definition}
Let $\phi$ be a {continuous metric integrand in} $\overline{\Omega}$. For a given function $u \in L^1(\Omega)$ we define its $\phi-$total variation in $\Omega$ by the formula:

\begin{equation*}
\int_\Omega |Du|_\phi = \sup \, \{ \int_\Omega u \, \mathrm{div} \, \mathbf{z} \, dx : \, \phi^0(x,\mathbf{z}(x)) \leq 1 \, \, \, \text{a.e.}, \quad \mathbf{z} \in C_c^1(\Omega)  \}.
\end{equation*}
Another popular notation for the $\phi-$total variation is $\int_\Omega \phi(x, Du)$. We will say that $u \in BV_\phi(\Omega)$ if its $\phi-$total variation is finite {in $\overline{\Omega}$}; furthermore, let us define the $\phi-$perimeter of a set $E$ as
$$P_\phi(E, \Omega) = \int_{\Omega} |D\chi_E|_\phi.$$
If $P_\phi(E, \Omega) < \infty$, we say that $E$ is a set of bounded $\phi-$perimeter in $\Omega$.
\end{definition}

\begin{remark}
By property $(3)$ of a metric integrand
$$\lambda \int_\Omega |Du| \leq \int_\Omega |Du|_\phi \leq \Lambda \int_\Omega |Du|.$$
In particular, $BV_\phi(\Omega) = BV(\Omega)$ as sets; however, they are equipped with different (but equivalent) norms. Moreover, the $\phi-$total variation admits the following integral representation:
\begin{equation*}
\int_\Omega |Du|_\phi = \int_\Omega \phi(x, \nu^u(x)) \, |Du|,
\end{equation*}
where $\nu^u$ is the Radon-Nikodym derivative $\nu^u = \frac{d Du}{d |Du|}$. If we take $u$ to be a characteristic function of a set $E$ with a $C^1$ boundary, we have
\begin{equation*}
P_\phi(E, \Omega) = {\int_{\partial E \cap \Omega}} \phi(x, \nu_E) \, d \mathcal{H}^{N-1},
\end{equation*}
where $\nu_E(x)$ is the (Euclidean) unit vector normal to $\partial E$ at $x \in \partial E$. {For the isotropic version of these facts, see \cite{AFP} or \cite{EG}; for the exposition of BV theory in the anisotropic setting and the integral representation formula, see \cite{AB}.}
\end{remark}

\subsection{$\phi-$least gradient functions}

Now, we turn our attention to the precise formulation of Problem (\ref{aproblem}). Then we recall several known properties of the minimisers.

\begin{definition}
Let $\Omega \subset \mathbb{R}^N$ be an open bounded set with Lipschitz boundary. We say that $u \in BV_\phi(\Omega)$ is a function of $\phi-$least gradient, if for every compactly supported $v \in BV_\phi(\Omega)$ we have

\begin{equation*}
\int_\Omega |Du|_\phi \leq \int_\Omega |D(u + v)|_\phi.
\end{equation*}
We say that $u$ is a solution to Problem (\ref{aproblem}), the anisotropic least gradient problem with boundary data $f$, if $u$ is a function of $\phi-$least gradient and $Tu = f$.
\end{definition}

We will recall a few properties of functions of $\phi-$least gradient. Firstly, we state an anisotropic version of Theorem \ref{tw:bgg}. Its proof in both directions is based on the the co-area formula.

\begin{theorem}\label{thm:anisobgg}
(\cite[Theorem 3.19]{Maz}) Let $\Omega \subset \mathbb{R}^N$ be an open bounded set with Lipschitz boundary. Assume that the metric integrand $\phi$ has a continuous extension to $\mathbb{R}^N$. Take $u \in BV_\phi(\Omega)$. Then $u$ is a function of $\phi-$least gradient in $\Omega$ if and only if $\chi_{\{ u > t \}}$ is a function of $\phi-$least gradient {in $\Omega$} for almost all $t \in \mathbb{R}$. \qed
\end{theorem}

Finally, we recall the following observation:

\begin{lemma}\label{lem:variationestimate}
(\cite[Lemma 2.16]{Gor4}) Let $\Omega \subset \mathbb{R}^N$ be an open bounded set with Lipschitz boundary. Let $u \in BV_\phi(\Omega)$ be a function of $\phi-$least gradient in $\Omega$. Then
$$ \int_\Omega |Du|_\phi \leq \int_{\partial\Omega} {\phi(x, \nu^\Omega)} |Tu| \, d\mathcal{H}^{N-1},$$
{where $\nu^\Omega$ is the $\mathcal{H}^{N-1}$-a.e. well-defined outer normal to $\partial\Omega$ at $x$.} \qed
\end{lemma}

\subsection{Monotonicity formula}

We will also use the monotonicity formula for stationary varifolds; we refer to \cite[\S 17]{Sim} for the full statement. Here, we will only use it in codimension one for area-minimising boundaries. In particular, we know that if $E$ is a minimal set in $\Omega$ (i.e. $\chi_E$ is a function of least gradient), then $\partial E$ is regular except for a set of Hausdorff dimension $N - 7$, hence $P(E, \Omega) =  \mathcal{H}^{N-1}(\partial E)$.

\begin{proposition}\label{prop:monotonicity}
Let $E \subset \Omega$ be a set of finite perimeter such that $\partial E$ is locally area-minimising. Then, for each $x \in \Omega$ and $r < \text{dist}(x, \partial \Omega)$, the function 
$$ f(x,r) = \frac{\mathcal{H}^{N-1}(\partial E \cap B(x,r))}{\omega_{N-1} r^{N-1}}$$
is increasing in $r$. In particular, the limit density
$$ \Theta(D\chi_E, x) = \lim_{r \rightarrow 0^+} f(x,r)$$
exists and equals at least one at {\it each} point of the support of $D\chi_E$. \qed
\end{proposition}

\subsection{Traces of characteristic functions}

The following Lemma is an easy exercise in traces of $BV$ functions, but to the best of the author's knowledge there is no proof in the literature.

\begin{lemma}\label{lem:traceforalmostallt}
Let $u \in BV(\Omega)$ and $Tu = f$. For all except countably many $t \in \mathbb{R}$ we have
$$ T\chi_{\{ u \geq t \}} = \chi_{\{ f \geq t \}}.$$
\end{lemma}

\begin{proof}
Denote $E_t = \{ u \geq t \}$. Fix $t \in \mathbb{R}$ such that $\mathcal{H}^{N-1}(\{ f = t \}) = 0$ (this happens for all but countably many $t$). For $\mathcal{H}^{N-1}$-almost every $x \in \partial\Omega$ we have
$$ \dashint_{\Omega \cap B(x,r)} |u(y) - f(x)| \, dy \rightarrow 0 \qquad \text{ as } r \rightarrow 0$$
and
$$ \dashint_{\Omega \cap B(x,r)} |\chi_{E_t}(y) - T\chi_{E_t}(x)| \, dy \rightarrow 0 \qquad \text{ as } r \rightarrow 0;$$
we denote the set of such points by $\mathcal{Z}$. By our assumption on $t$, the set $\mathcal{Z} \cap \{ f \neq t\}$ is of full measure. Now, fix $x \in \mathcal{Z} \cap \{ f \neq t\}$.

There are two possibilities: either $x \in \{ f > t \}$ or $x \in \{ f < t \}$. Without loss of generality assume that $f(x) = s > t$. Suppose that $T\chi_{E_t}(x) \neq 1 = \chi_{\{ f > t \}}(x)$. Then on a subsequence $r_n \rightarrow 0$ we have
$$ \dashint_{\Omega \cap B(x,r_n)} |\chi_{E_t}(y) - 1| \, dy \geq c. $$
We rewrite this as
$$ c \leq \dashint_{\Omega \cap B(x,r_n)} |\chi_{E_t}(y) - 1| \, dy = \dashint_{\Omega \cap B(x,r_n)} |\chi_{\Omega \backslash E_t}(y)| \, dy = \frac{|(\Omega \cap B(x, r_n)) \backslash E_t|}{|\Omega \cap B(x, r_n)|}. $$
Now, we see that this leads to a contradiction with $Tu(x) = f(x) = s$. We calculate
$$ \dashint_{\Omega \cap B(x,r_n)} |u(y) - f(x)| \, dy = \frac{1}{|\Omega \cap B(x, r_n)|} \int_{\Omega \cap B(x,r_n)} |u(y) - s| \, dy \geq$$
$$ \geq \frac{1}{|\Omega \cap B(x, r_n)|} \int_{(\Omega \cap B(x,r_n)) \backslash E_t} |u(y) - s| \, dy \geq $$
$$ \geq \frac{1}{|\Omega \cap B(x, r_n)|} \int_{(\Omega \cap B(x,r_n)) \backslash E_t} |t - s| \, dy = $$
$$ = \frac{|(\Omega \cap B(x,r_n)) \backslash E_t|}{|\Omega \cap B(x, r_n)|}  |t - s| \geq c (t - s) > 0,$$
hence there exists a sequence $r_n \rightarrow 0$ such that the mean integral condition defining the trace of $u$ at $x$ is not satisfied, contradiction. Thus $T\chi_{E_t}(x) = 1 = \chi_{\{ f > t \}}(x)$.
\end{proof}

\section{$L^p$ regularity}

In this Section, we prove $L^{\frac{Np}{N-1}}$ regularity of least gradient functions for $L^p$ boundary data. The exponent we obtain is consistent with the exponent in the inclusion $BV(\Omega) \subset L^p(\Omega)$ for $p \leq \frac{N}{N-1}$; at the end of the Section, we provide an example that this estimate is optimal. The following Theorem is valid without any regularity assumptions on the metric integrand $\phi$. 
The first result in this Section is an estimate on the Lebesgue measure of a superlevel set of a function of $\phi$-least gradient. Then, we will prove that this estimate implies Theorem \ref{thm:lpregularity}, which is the main result in this Section.

\begin{lemma}\label{lem:measureestimate}
Suppose that $\Omega \subset \mathbb{R}^N$ is an open bounded set with Lipschitz boundary. Let $u \in BV(\Omega)$ be a $\phi$-least gradient function such that $Tu = f$. Then for almost all $t \in \mathbb{R}$ we have
$$ |\{ u \geq t \}| \leq C(\phi, N) (\mathcal{H}^{N-1}(\{ f \geq t \}))^{\frac{N}{N-1}}.$$

\end{lemma}

\begin{proof}
Denote $E_t = \{ u \geq t \}$. Recall the isoperimetric inequality: if $E \subset \mathbb{R}^N$ is a bounded set of finite perimeter, then (see for instance \cite[Theorem 5.6.2]{EG})
$$ |E|^{\frac{N-1}{N}} \leq C_N P(E, \mathbb{R}^N).$$
We want to use the isoperimetric inequality to estimate the Lebesgue measure of the set $E_t$. To this end, as $E_t$ is defined as a superlevel set of $u$ and hence a subset of $\Omega$, we firstly have to estimate $P(E, \mathbb{R}^N)$. We recall that (see for instance \cite[Theorem 5.4.1]{EG}) if $\Omega$ is an open bounded set with Lipschitz boundary and $u_1 \in BV(\Omega)$ and $u_2 \in BV(\mathbb{R}^N \backslash \overline{\Omega})$, then the extension $\widetilde{u} = u_1 \chi_\Omega + u_2 \chi_{\mathbb{R}^N \backslash \overline{\Omega}}$ lies in $BV(\mathbb{R}^N)$ and
$$ \int_{\mathbb{R}^N} |D\widetilde{u}| = \int_\Omega |Du_1| + \int_{\mathbb{R}^N \backslash \overline{\Omega}} |Du_2| + \int_{\partial \Omega} |Tu_1 - Tu_2| d\mathcal{H}^{N-1}. $$
We use this result with $u_1 = \chi_{E_t}$ and $u_2 = 0$ to estimate $P(E_t, \mathbb{R}^N)$. For almost all $t$, so that the statements of Theorem \ref{thm:anisobgg} and Lemma \ref{lem:traceforalmostallt} hold, we calculate
$$ P(E_t, \mathbb{R}^N) = P(E_t, \Omega) + 0 + \int_{\partial\Omega} |T\chi_{E_t}| d\mathcal{H}^{N-1} = P(E_t, \Omega) + \mathcal{H}^{N-1}(\{ f \geq t \}),$$
where the last equality follows from Lemma \ref{lem:traceforalmostallt}. Now, we recall that by Theorem \ref{thm:anisobgg} $\chi_{E_t}$ is a function of $\phi$-least gradient for almost all $t \in \mathbb{R}$. By Lemma \ref{lem:variationestimate}
$$ P(E_t, \Omega) \leq \lambda^{-1} P_\phi(E_t, \Omega) = \lambda^{-1} \int_\Omega |D\chi_{E_t}|_\phi \leq \lambda^{-1} \int_{\partial\Omega} \phi(x, \nu^\Omega) |T\chi_{E_t}| \leq$$
$$ \leq \lambda^{-1} \Lambda \int_{\partial\Omega} |T\chi_{E_t}| = \lambda^{-1} \Lambda \mathcal{H}^{N-1}(\{ f \geq t \}),$$
where in the last equality we use Lemma \ref{lem:traceforalmostallt}. Hence
$$ P(E_t, \mathbb{R}^N)\leq (\lambda^{-1} \Lambda + 1) \mathcal{H}^{N-1}(\{ f \geq t \}) $$
and by isoperimetric inequality we obtain
$$ |E_t|^{\frac{N-1}{N}} \leq C_N (\lambda^{-1} \Lambda + 1) \mathcal{H}^{N-1}(\{ f \geq t \}).$$
We take both sides of this inequality to the power $\frac{N}{N-1}$ to obtain the desired inequality with $C(\phi, N) = (C_N (\lambda^{-1} \Lambda + 1))^{\frac{N}{N-1}}$.
\end{proof}

\begin{theorem}\label{thm:lpregularity}
Suppose that $\Omega \subset \mathbb{R}^N$ is an open bounded set with Lipschitz boundary. Suppose that $1 \leq p < \infty$. Let $u \in BV(\Omega)$ be a $\phi$-least gradient function such that $Tu = f \in L^p(\partial\Omega)$. Then $u \in L^{\frac{Np}{N-1}}(\Omega)$.
\end{theorem}
 
\begin{proof}
Denote $q = \frac{Np}{N-1}$. Let us decompose $u$ into a positive and negative part, i.e. $u = u_+ - u_-$, where $u_+ = \max(u,0)$ and $u_- = \max(-u,0)$. Let $f = f_+ - f_-$ be an analogous decomposition for $f$. We will prove that $u_+ \in L^q(\Omega)$ and at the end remark how to modify this proof to show that also $u_- \in L^q(\Omega)$.

Firstly, we recall that for any measure space $(X, \mu)$ we have an inclusion $L^p(X, \mu) \subset L^p_w(X, \mu)$, where $L^p_w(X, \mu)$ denotes the weak Lebesgue space, and that the seminorm $\| g \|_{L_w^p(X, \mu)}$ is bounded by the norm $\| g \|_{L^p(X, \mu)}$. In other words, for all $t > 0$
\begin{equation*}
\mu(\{ |g| \geq t \}) \leq \frac{\| g \|_{L^p(X, \mu)}^p}{t^p}.
\end{equation*}
We apply this to $(X, \mu) = (\partial\Omega, \mathcal{H}^{N-1})$, $g = f_+ \in L^p(\partial\Omega)$ and take both sides of the inequality to power $\frac{1}{N-1}$ to obtain that for all $t > 0$
\begin{equation*}\tag{*}
(\mathcal{H}^{N-1}(\{ f_+ \geq t \}))^{\frac{1}{N-1}} \leq \frac{\| f_+ \|_{L^p(\partial\Omega)}^\frac{p}{N-1}}{t^{\frac{p}{N-1}}}.
\end{equation*}
Denote $E_t = \{ u \geq t \}$. While $u_+$ is not necessarily a function of $\phi-$least gradient, we immediately see that for $t > 0$ we have $\{ u_+ \geq t \} = E_t$; hence, by Theorem \ref{thm:anisobgg} almost every superlevel set of $u_+$ is area-minimising. We calculate
$$ \int_\Omega (u_+)^q dx = q \int_0^\infty t^{q-1} |\{ u_+ \geq t \}| dt = q \int_0^\infty t^{q-1} |E_t| dt. $$
Now, we use Lemma \ref{lem:measureestimate} to estimate the last integral. 
$$ q \int_0^\infty t^{q-1} |E_t| dt \overset{(3.1)}{\leq} q \int_0^\infty t^{q-1} C(\phi, N) (\mathcal{H}^{N-1}(\{ f \geq t \}))^{\frac{N}{N-1}} dt = $$
$$ = q \int_0^\infty t^{q-1} C(\phi, N) (\mathcal{H}^{N-1}(\{ f_+ \geq t \}))^{\frac{N}{N-1}} dt = $$
$$ = C(\phi, N) \, q \int_0^\infty t^{q-1} (\mathcal{H}^{N-1}(\{ f_+ \geq t \}))^{\frac{N}{N-1}} \mathcal{H}^{N-1}(\{ f_+ \geq t \}) dt \overset{(*)}{\leq} $$
$$ \overset{(*)}{\leq} C(\phi, N) \, q \int_0^\infty t^{q-1} \frac{\| f_+ \|_{L^p(\partial\Omega)}^\frac{p}{N-1}}{t^{\frac{p}{N-1}}} \mathcal{H}^{N-1}(\{ f_+ \geq t \}) dt = $$
$$ = C(\phi, N) \| f_+ \|_{L^p(\partial\Omega)}^\frac{p}{N-1} \, q \int_0^\infty t^{q-1 - \frac{p}{N-1}} \mathcal{H}^{N-1}(\{ f_+ \geq t \}) dt = $$
$$ = C(\phi, N) \| f_+ \|_{L^p(\partial\Omega)}^\frac{p}{N-1} \, \frac{q}{p} p \int_0^\infty t^{p-1} \mathcal{H}^{N-1}(\{ f_+ \geq t \}) dt = $$
$$ = \frac{N}{N-1} C(\phi, N) \| f_+ \|_{L^p(\partial\Omega)}^\frac{p}{N-1} \int_{\partial\Omega}  (f_+)^p d\mathcal{H}^{N-1}.$$
We combine the above estimates to obtain
$$ \| u_+ \|_{L^q(\Omega)}^q \leq \frac{N}{N-1} C(\phi, N) \| f_+ \|_{L^p(\partial\Omega)}^{\frac{p}{N-1} + p}.$$
We take both sides to power $\frac{1}{q}$ and obtain
$$ \| u_+ \|_{L^q(\Omega)} \leq (\frac{N}{N-1} C(\phi, N))^{\frac{N-1}{Np}} \| f_+ \|_{L^p(\partial\Omega)}.$$
Hence, if $f_+ \in L^p(\partial\Omega)$, then $u_+ \in L^q(\Omega)$. Now, we make a similar calculation for $u_-$: we take $\widetilde{E_t} = \{ u \leq t \}$ for $t < 0$. We easily see that $\{ u_- \geq -t \} = \widetilde{E_t}$ and we proceed as above, except for the fact that we use a Lemma \ref{lem:measureestimate} for $-u$ in place of $u$ to estimate the measure of $\widetilde{E_t}$. Finally, as $u_+ \in L^q(\Omega)$ and $u_- \in L^q(\Omega)$, we have that $u \in L^q(\Omega)$.
\end{proof}

\begin{remark}
As the proof of Theorem \ref{thm:lpregularity} comes in two parts in which we estimate separately the $L^{\frac{Np}{N-1}}$ norm of $u_\pm$ in terms of the $L^p$ norm of $f_\pm$, a following variant of the Theorem holds: in the notation of Theorem \ref{thm:lpregularity}, let $f_\pm \in L^p(\partial\Omega)$. Then $u_\pm \in L^{\frac{Np}{N-1}}(\Omega)$.
\end{remark}

\begin{remark}\label{rmk:maximumprinciple}
Notice that the estimate on the norm of $u$ does not depend on $\Omega$, only on the dimension (both directly and via the constant in the isoperimetric inequality) and the bounds on the metric integrand $\phi$. Moreover, we see that if we let $p \rightarrow \infty$, we obtain exactly the maximum principle for least gradient functions (see for instance \cite[Theorem 5.1]{HKLS}):
$$ \| u_+ \|_{L^\infty(\Omega)} \leq \| f_+ \|_{L^\infty(\partial\Omega)}.$$
\end{remark}

Finally, we present an example that showing that the exponent $\frac{Np}{N-1}$ in Theorem \ref{thm:lpregularity} is optimal.

\begin{example}\label{ex:optimal}
Let $\Omega = \{ (x,y): |x-1| + |y| \leq 1 \} \subset \mathbb{R}^2$. Take $f(x,y) = g(x)$, where $g \in L^1((0,2)) \cap C((0,2))$ is a decreasing function such that $g(x) = 1$ on $[1,2)$.  Then the function $u(x,y) = g(x)$, i.e. such that all level lines are vertical, is a function of least gradient with trace $f$.

Now, we look at the measure of superlevel sets of $u$. For all $t > 1$, $\{ u \geq t \}$ is a triangle with vertices $(0,0), (g^{-1}(t), g^{-1}(t))$ and $(g^{-1}(t), - g^{-1}(t))$, so
$$ |\{ u \geq t \}| = (g^{-1}(t))^2.$$
Let $p \geq 1$. We use this estimate to calculate
$$ \int_\Omega u^p dx = p \int_0^\infty t^{p-1} | \{ u \geq t \} |dt = p \int_0^\infty t^{p-1} (g^{-1}(t))^2 dt \geq p \int_1^\infty t^{p-1} (g^{-1}(t))^2 dt.$$
Now, we fix a function $g_n$ defined by the formula $g_n(x) = x^{-1 + \frac{1}{n}}$. We see that $g_n$ is continuous, strictly decreasing, and that $g(x) = 1$ on $[1,2)$. We put $g_n$ in the calculation above and obtain
$$ \int_\Omega (u_n)^p dx \geq p \int_1^\infty t^{p-1} t^{-\frac{2n}{n-1}} dt = p \int_1^\infty t^{p-\frac{2n}{n-1}-1} dt,$$
and the last integral is finite if and only if $p < \frac{2n}{n-1}$. We pass with $n \rightarrow \infty$ and see that the statement of Theorem \ref{thm:lpregularity} can only hold for $p \leq 2$, which is precisely the exponent given by Theorem \ref{thm:lpregularity}.
\end{example}

\section{$L^\infty_{loc}$ regularity}

When the boundary data $f$ lie in $L^\infty(\partial\Omega)$, then the maximum principle as in Remark \ref{rmk:maximumprinciple} implies that any solution $u$ to Problem \eqref{aproblem} lies in $L^\infty(\Omega)$. Conversely, if $f \notin L^\infty(\partial\Omega)$, then $u \notin L^\infty(\Omega)$, as the trace of a bounded function cannot be unbounded. However, it turns out that $u$ may blow up only near the boundary of $\Omega$.

This Section contains three versions of the result stating that $\phi$-least gradient functions are locally bounded. Firstly, we prove this result on a toy model: we assume that $\Omega \subset \mathbb{R}^2$ and that $\phi$ is the Euclidean norm. Then, in Proposition \ref{prop:locallybounded} we prove this in $\Omega \subset \mathbb{R}^2$ for any metric integrand $\phi$, using a characterisation of one-dimensional integral currents. Finally, in Theorem \ref{thm:liniftyloc} we prove this in any dimension for the isotropic least gradient problem, using the monotonicity formula for area-minimising boundaries.

\begin{proposition}
Let $\Omega \subset \mathbb{R}^2$ be an open bounded set with Lipschitz boundary. Suppose that $u$ is a least gradient function. Then $u \in L^\infty_{loc}(\Omega)$.
\end{proposition}

\begin{proof}
Denote $E_t = \{ u \geq t \}$. Let $\Omega' \subset \subset \Omega$ be open with Lipschitz boundary and suppose that $u \notin L^\infty(\Omega')$. Without loss of generality $u$ is unbounded from above. In particular, for each $t > 0$ we have $|E_t \cap \Omega'| > 0$. As $u \in L^1(\Omega)$, for sufficiently large $t \geq M$ we have $|E_t \cap \Omega'| \neq |\Omega'|$, hence $\partial E_t \cap \Omega' \neq \emptyset$.

As by Theorem \ref{tw:bgg} each connected component of $\partial E_t$ is a line segment with ends on $\partial\Omega$, the connected component of $E_t$ passing through $\Omega'$ has length equal at least to $\text{dist}(\partial\Omega, \partial\Omega')$. By the co-area formula
$$ \int_\Omega |Du| = \int_\mathbb{R} P(E_t, \Omega) \geq \int_M^\infty P(E_t, \Omega) \geq \int_M^\infty \text{dist}(\partial\Omega, \partial\Omega') = +\infty, $$
contradiction with $u \in BV(\Omega)$.
\end{proof}

\begin{proposition}\label{prop:locallybounded}
Let $\Omega \subset \mathbb{R}^2$ be an open bounded set with Lipschitz boundary. Suppose that $\phi$ is a metric integrand and that $u$ is a $\phi$-least gradient function. Then $u \in L^\infty_{loc}(\Omega)$.
\end{proposition}

\begin{proof}
Let $A \subset \Omega$ be a set of finite perimeter. As $\Omega \subset \mathbb{R}^2$, the measure $D\chi_A$ is a one-dimensional integral current. By \cite[\S 4.2.25]{Fed}, each one-dimensional integral current may be decomposed into a (possibly infinite) sum of indecomposable integral currents. Each such current  $T$ is an oriented simple curve with finite length, i.e. its support is parametrised by a function $h: \mathbb{R} \rightarrow \mathbb{R}^N$ with Lip$(h) \leq 1$ and $f_{\#}((0,  \mathbb{M}(T))) = T$.

Let $\Omega' \subset \subset \Omega$ be open with Lipschitz boundary and suppose that $u \notin L^\infty(\Omega')$. Without loss of generality $u$ is unbounded from above. In particular, for each $t > 0$ we have $|E_t \cap \Omega'| > 0$. As $u \in L^1(\Omega)$, for sufficiently large $t \geq M$ we have $|E_t \cap \Omega'| \neq |\Omega'|$, hence $\partial E_t \cap \Omega' \neq \emptyset$.

Let $E_t$ be as above, hence it is a $\phi-$minimal set. Then $\partial^* E_t$, the reduced boundary of $E_t$, can be represented (up to a set of $\mathcal{H}^1$-measure zero) as a possibly infinite union of Lipschitz curves. We have
$$ \partial^* E_t \cup N = \bigcup_i \Gamma_i,$$
where $\Gamma_i$ are Lipschitz curves and $\mathcal{H}^1(N) = 0$. As $E_t$ is a $\phi$-minimal set, none of these curves are closed loops. Without loss of generality, assume that $x \in \Gamma_i$ for some $i \in \mathbb{N}$; if $x \in N$, then (with our convention of representing sets of finite perimeter) we could replace it by a point in some $\Gamma_i$ arbitrarily close to $x$. Now, we notice that the Euclidean length of a Lipschitz curve connecting $x$ and a point in $\partial\Omega$ is at least $\text{dist}(\partial\Omega, \partial\Omega')$  and estimate
$$ P_\phi(E_t, \Omega) = \int_{\partial^* F} \phi(x, \nu(x)) d\mathcal{H}^1 = \sum_{i = 1}^\infty \int_{\Gamma_i} \phi(x, \nu(x)) d\mathcal{H}^1 \geq $$
$$= \int_{\Gamma_i} \phi(x, \nu(x)) d\mathcal{H}^1 \geq \int_{\Gamma_i} \lambda |\nu(x)| d\mathcal{H}^1 = \lambda \mathcal{H}^1(\Gamma_i) \geq \lambda \, \text{dist}(\partial\Omega, \partial\Omega'),$$
hence we have a uniform bound from below. By the co-area formula
$$ \int_\Omega |Du| = \int_\mathbb{R} P(E_t, \Omega) \geq \int_M^\infty P(E_t, \Omega) \geq \int_M^\infty \text{dist}(\partial\Omega, \partial\Omega') = +\infty, $$
contradiction with $u \in BV(\Omega)$.
\end{proof}

\begin{theorem}\label{thm:liniftyloc}
Let $\Omega \subset \mathbb{R}^N$ be an open bounded set with Lipschitz boundary. Suppose that $u$ is a least gradient function. Then $u \in L^\infty_{loc}(\Omega)$.
\end{theorem}

\begin{proof}
We start as in the previous Propositions: denote $E_t = \{ u \geq t \}$, let $\Omega' \subset \subset \Omega$ be open with Lipschitz boundary and suppose that $u$ is unbounded from above. In particular, for sufficiently large $t \geq M$ we have $\partial E_t \cap \Omega' \neq \emptyset$.

As previously, we intend to use the co-area formula and we need to estimate from below the perimeter of $E_t$. To this end, we will use Proposition \ref{prop:monotonicity}, i.e. the monotonicity formula. We recall that as $E_t$ are area-minimising, we have $P(E_t, \Omega) = \mathcal{H}^{N-1}(\partial E_t)$. Now, let us fix $x \in \partial E_t \cap \Omega'$. Then, in the notation of Proposition \ref{prop:monotonicity}, the density $\Theta$ satisfies $\Theta(D\chi_{E_t},x) \geq 1$, hence $f(x,r) \geq 1$ for $r < \text{dist}(x, \partial\Omega)$. We set $r = \frac{\text{dist}(\partial \Omega, \partial\Omega')}{2}$ and calculate
$$ \int_{\Omega} |Du| = \int_\mathbb{R} P(E_t, \Omega) dt \geq \int_{M}^\infty P(E_t, \Omega) dt = $$
$$ = \int_{M}^\infty \mathcal{H}^{N-1}(\partial E_t) dt \geq  \int_{M}^\infty \mathcal{H}^{N-1}(\partial E_t \cap B(x_t, \frac{\text{dist}(\partial\Omega, \partial\Omega')}{2})) dt \geq $$
$$ \geq \int_{M}^\infty \omega_{N-1} (\frac{\text{dist}(\partial\Omega, \partial\Omega')}{2})^{N-1} dt = + \infty, $$
contradiction with $u \in BV(\Omega)$.
\end{proof}

In fact, this proof leads to an explicit bound on the essential range on $u$ on $\Omega'$, which depends on $\| Tu \|_{L^1(\partial\Omega)}$ and $\text{dist}(\partial\Omega, \partial\Omega')$. Before we state Corollary \ref{cor:blowuprate}, let us notice that
$$ \text{ess} \, \sup_{\Omega'} u = \sup_{t: \partial E_t \cap \Omega' \neq \emptyset} t$$
and
$$ \text{ess} \, \inf_{\Omega'} u = \inf_{t: \partial E_t \cap \Omega' \neq \emptyset} t.$$

\begin{corollary}\label{cor:blowuprate}
Let $\Omega \subset \mathbb{R}^N$ be an open bounded set with Lipschitz boundary and $\Omega' \subset \subset \Omega$ be an open bounded set with Lipschitz boundary. Suppose that $u$ is a least gradient function with trace $Tu = f \in L^1(\partial\Omega)$. Then 
$$ \text{ess} \, \sup_{\Omega'} u - \text{ess} \, \inf_{\Omega'} u \leq \frac{C(N) \| f \|_{L^1(\partial\Omega)}}{(\text{dist}(\partial\Omega, \partial\Omega'))^{N-1}}. $$
The left hand side of the above inequality describes the width of the essential range of $u$ on $\Omega'$. 
\end{corollary}

\begin{proof}
By Lemma \ref{lem:variationestimate} we have $\int_\Omega |Du| \leq \int_{\partial\Omega} |Tu| d\mathcal{H}^{N-1} = \| f \|_{L^1(\partial\Omega)}$. We make a similar calculation as in the proof of Theorem \ref{thm:liniftyloc} and we see that
$$ \| f \|_{L^1(\partial\Omega)} \geq \int_{\Omega} |Du| = \int_\mathbb{R} P(E_t, \Omega) dt \geq \int_{\text{ess} \, \inf_{\Omega'} u}^{\text{ess} \, \sup_{\Omega'} u} P(E_t, \Omega) dt \geq $$
$$ \geq \int_{\text{ess} \, \inf_{\Omega'} u}^{\text{ess} \, \sup_{\Omega'} u} \omega_{N-1} (\frac{\text{dist}(\partial\Omega, \partial\Omega')}{2})^{N-1} dt = $$
$$ = \frac{\omega_{N-1}}{2^{N-1}} (\text{ess} \, \sup_{\Omega'} u - \text{ess} \, \inf_{\Omega'} u) (\text{dist}(\partial\Omega, \partial\Omega'))^{N-1},$$
from which follows the desired inequality with constant $C(N) = \frac{2^{N-1}}{\omega_{N-1}}$.
\end{proof}

{\bf Acknowledgements.} I would like to thank my PhD advisor, Piotr Rybka, for his support and fruitful discussions about this paper. This work was partly supported by the research project no. 2017/27/N/ST1/02418, "Anisotropic least gradient problem", funded by the National Science Centre, Poland.

\bibliographystyle{plain}
\bibliography{WG-references}


\bibliographystyle{amsplain}

\end{document}